\documentclass{article}
\usepackage{amsmath,amssymb,amsthm}
\usepackage{amscd}
\usepackage{verbatim}
\usepackage[english]{babel}
\usepackage[linktocpage]{hyperref}
\usepackage{titlesec}
\usepackage[rightcaption]{sidecap}

\usepackage{wrapfig}
\usepackage{blindtext}
\newtheorem{theorem}{Theorem}
\newtheorem{lemma}{Lemma}
\newtheorem{remark}[theorem]{Remark}
\newtheorem{proposition}[theorem]{Proposition}
\newtheorem{definition}[theorem]{Definition}

\newtheorem{corollary}[theorem]{Corollary}
\newtheorem{problem}[theorem]{Problem}

\title{Cyclic adjoint modules and their embeddings in quantized enveloping algebras}
\author{Arnab Bhattacharjee \thanks{Mathematical Institute of Charles University, email: arnabbhatta7@gmail.comThe author acknowledges support from HORIZON-MSCA-2021-SE-01-CaLIGOLA and also from COST Action CaLISTA CA21109.}}
\date{}
\begin{document}
\maketitle

\begin{abstract}
We study cyclic adjoint modules arising from the relative locally finite part of the adjoint action of a quantum Levi subalgebra on a quantized enveloping algebra. We classify embeddings of finite-dimensional irreducible modules inside of quantized enveloping algebra via cyclic generators and show that such realizations are in general non-unique, exhibiting infinite families in the cominuscule case. We also introduce a partial order on cyclic adjoint modules, characterize its minimal elements, and prove finite generation by irreducible submodules.
\end{abstract}

\tableofcontents

\section{Introduction}

The adjoint action of a quantized enveloping algebra \cite{drinfeld1986quantum, jimbo1985q, lusztig2010introduction} on itself plays a central role in understanding its internal structure and its relation to representation theory. A fundamental result of Joseph and Letzter \cite{joseph1992local, letzter2002coideal} shows that the locally finite part of this action admits a rich algebraic description. In many situations, it is natural to study such structures relative to a quantum Levi subalgebra $U_q(\mathfrak l_S) \subset U_q(\mathfrak g)$ (see \cite{joseph2012quantum}). This leads to the relative locally finite part
\[
F_S := \{ v \in U_q(\mathfrak g) \mid \dim \operatorname{ad}_L(U_q(\mathfrak l_S)) \cdot v < \infty \},
\]
which provides a natural setting for realizing finite-dimensional $U_q(\mathfrak l_S)$-modules inside $U_q(\mathfrak g)$.

The main problem addressed in this paper is the following:
\medskip

\noindent
\begin{problem}
     How do finite-dimensional irreducible $U_q(\mathfrak l_S)$-modules arise inside $U_q(\mathfrak g)$ via the adjoint action, and to what extent are such realizations unique?
\end{problem}

To study this question, we associate to each $v \in F_S$ the cyclic adjoint module
\[
M(v) := \operatorname{ad}_L(U_q(\mathfrak l_S)) \cdot v,
\]
and analyze embeddings of a fixed type $\lambda$, i.e.\ $U_q(\mathfrak l_S)$-submodules of $U_q(\mathfrak g)$ isomorphic to $V(\lambda)$. We show that embeddings of type $\lambda$ are parametrized by elements of $F_S(\lambda)$ modulo equality of generated submodules, leading to a natural map
\[
\Phi : F_S \to \{\text{isomorphism classes of finite-dimensional } U_q(\mathfrak l_S)\text{-modules}\},
\quad v \mapsto [M(v)].
\]
A key feature of this framework is the non-uniqueness of realizations: using results of Krähmer \cite[Proposition 1]{krahmer2004dirac} in the cominuscule case, we exhibit infinite families of distinct elements of $F_S$ generating isomorphic irreducible modules, showing that the fibers of $\Phi$ can be infinite.

In addition, we investigate structural properties of cyclic adjoint modules. We introduce a preorder on $F_S$ given by
\[
w \le v \quad \Longleftrightarrow \quad M(w) \subseteq M(v),
\]
and study the induced partially ordered set
\[
\mathcal P_S := \{ M(v) \mid v \in F_S \}.
\]
Within this framework, irreducible modules correspond to minimal elements, and every cyclic adjoint module is generated by finitely many irreducible submodules arising from highest weight vectors.

These results provide a concise framework for understanding the realization and structure of cyclic adjoint modules inside $U_q(\mathfrak g)$, highlighting both the existence and the non-uniqueness of such embeddings.

\section{Preliminaries}
Throughout the paper $\mathfrak{g}$ is a complex semisimple Lie algebra and $\mathfrak{h}$ denotes the cartan subalgebra of $\mathfrak{g}$. In this section we follow literature from text books \cite{joseph2012quantum, klimyk2012quantum}.

Let $\lbrace \alpha_1, \alpha_2, \cdots , \alpha_n \rbrace$ and $\lbrace \varpi_1, \varpi_2, \cdots , \varpi_n \rbrace$ are a set of simple roots and the corresponding set of fundamental weights of Lie algebra $\mathfrak{g}$. The integral roots and weight lattices are denoted by $Q$ and $P$ respectively and the dominant integral weight are denoted by $P^+$. The killing form induces a bilinear pairing $\langle \cdot, \cdot \rangle$ on $P\times P$, and $\langle \varpi_{i}, \alpha_j \rangle= \delta_{ij}\, d_i$.

\begin{definition}
    Let $U_{q}(\frak{g})$ be the \emph{quantized enveloping algebra} corresponding to $\mathfrak{g}$ with generators $E_{i}, F_{i}, K_{i}$ and $K^{-1}_{i}$. See (\cite{joseph2012quantum}, \S 3.2.9) for their generating relations explicitly. For the coproduct, counit and antipode, we use the conventions of \cite{joseph2012quantum}.
    \[
    \Delta(E_i)= E_i\otimes 1+ K_i \otimes E_i, \,\,\,\,\, \Delta(F_i)= F_i\otimes K^{-1}_i+ 1\otimes F_i 
    \]
\end{definition}

 \begin{definition}
    For $\lbrace \alpha_{i}\rbrace_{i\in S}$, a subset of simple roots of $\frak{g}$, we consider the Hopf subalgebra, $$U_{q}(\frak{l}_{S}):= \langle E_{i}, F_{i}, K_{j}^{\pm1}: i\in S, j=1,2,..,r\rangle$$ which we call the \emph{quantum Levi subalgebra}.
    \end{definition}

    Left adjoint action of $U_q(\mathfrak{g})$ on itself is given by 
    \[
    \operatorname{ad}_L(a) (b):= a_{(1)}\,b \, S(a_{(2)}) \,\,\,\, \text{for}\,\,\,\, a,\, b\in U_q(\mathfrak{g})
    \]
\begin{definition}
    We define \emph{relative locally finite part of $U_q(\mathfrak{g})$} under the left adjoint action of quantum Levi subalgebra, and is given by 
\[
F_S := \{ v \in U_q(\mathfrak{g}) \mid \dim \operatorname{ad}_L(U_q(\mathfrak{l}_S)) \cdot v < \infty \}.
\]

By denoting $\operatorname{ad}_L(U_q(\mathfrak{l}_S)) \cdot v$ we mean that $\operatorname{ad}_L(U_q(\mathfrak{l}_S)) \cdot v:= \lbrace \operatorname{ad}_L(a)(v): a\in U_q(\mathfrak{l}_S)\rbrace$. See (\cite{joseph2012quantum}, \S 1.3.1) and (\cite{joseph1992local}, \S 2.3) for more details. 
\end{definition}

\section{Realization and embeddings of cyclic adjoint modules}

\begin{definition}
    Let $v\in F_S$, a \emph{cyclic adjoint module} generated by $v$ under the left adjoint action of $U_q(\mathfrak{l}_S)$ is given by 
\[
M(v) := \operatorname{ad}_L(U_q(\mathfrak{l}_S)) \cdot v.
\]
\end{definition}

\begin{lemma}\label{lem:inclusion}
If $w \in M(v)$, then $M(w) \subseteq M(v)$.
\end{lemma}

\begin{proof}
Let $w = \operatorname{ad}_L(x)(v)$ for some $x \in U_q(\mathfrak{l}_S)$.
Then for any $y \in U_q(\mathfrak{l}_S)$,
\[
\operatorname{ad}_L(y)(w)
= \operatorname{ad}_L(yx)(v) \in M(v).
\]
Thus $M(w) \subseteq M(v)$.
\end{proof}

We consider $V(\lambda)$ as an irreducible $U_q(\mathfrak{l}_S)$ module for $\lambda\in P^+$, and define

\[
F_S(\lambda)
:=
\left\{\, v \in F_S \;\middle|\; M(v) \simeq V(\lambda) \,\right\}.
\]

\[
\Lambda_S
:=
\left\{\, \lambda \in \mathfrak{h}^* \;\middle|\; F_S(\lambda) \neq \varnothing \,\right\}.
\]

\begin{remark}
By the classification of finite-dimensional irreducible $U_q(\mathfrak l_S)$-modules (see, for example, \cite{joseph2012quantum, klimyk2012quantum}), any $\lambda \in \Lambda_S$ is a dominant integral weight. In particular, $\Lambda_S \subseteq P^+$.
\end{remark}

\begin{definition}
Let $\lambda \in P^+$. By an \emph{embedding of type $\lambda$} into $U_q(\mathfrak{g})$,
we mean a $U_q(\mathfrak{l}_S)$-submodule $M \subset U_q(\mathfrak{g})$
such that $M \simeq V(\lambda)$ as $U_q(\mathfrak{l}_S)$-modules.
\end{definition}

\begin{proposition}
Let $\lambda \in P^+$. For each $v \in F_S(\lambda)$, the subspace
\[
M(v) = \operatorname{ad}_L(U_q(\mathfrak{l}_S)) \cdot v
\]
defines an embedding of type $\lambda$ into $U_q(\mathfrak{g})$.
\end{proposition}

\begin{proof}
Let $v \in F_S(\lambda)$. By definition of $F_S(\lambda)$, we have
\[
M(v) \simeq V(\lambda)
\]
as $U_q(\mathfrak{l}_S)$-modules.

On the other hand, by construction, $M(v)$ is a subspace of $U_q(\mathfrak{g})$
and is stable under the adjoint action of $U_q(\mathfrak{l}_S)$, i.e.
\[
\operatorname{ad}_L(x)(M(v)) \subseteq M(v)
\quad \text{for all } x \in U_q(\mathfrak{l}_S).
\]
Thus $M(v)$ is a $U_q(\mathfrak{l}_S)$-submodule of $U_q(\mathfrak{g})$.

Since $M(v)$ is both a submodule of $U_q(\mathfrak{g})$ and isomorphic to $V(\lambda)$,
it defines an embedding of type $\lambda$ into $U_q(\mathfrak{g})$.
\end{proof}

\begin{lemma}\label{lem:equivalence relation}
Let $v, w \in F_S(\lambda)$. Then $v$ and $w$ determine the same embedding of type $\lambda$
if and only if
\[
M(v) = M(w).
\]
\end{lemma}

\begin{proof}
Suppose first that $M(v) = M(w)$. Then both $v$ and $w$ generate the same
$U_q(\mathfrak{l}_S)$-submodule of $U_q(\mathfrak{g})$. Hence they determine the same
embedding of type $\lambda$ by definition.

Conversely, suppose that $v$ and $w$ determine the same embedding.
By definition, this means that the submodules of $U_q(\mathfrak{g})$ generated by $v$
and $w$ coincide as $U_q(\mathfrak{l}_S)$-submodules. That is,
\[
\operatorname{ad}_L(U_q(\mathfrak{l}_S)) \cdot v
=
\operatorname{ad}_L(U_q(\mathfrak{l}_S)) \cdot w.
\]
Hence $M(v) = M(w)$.
\end{proof}

\begin{proposition}\label{prop:embedding bijection}
Let $\lambda \in P^+$. There is a natural bijection between the set of embeddings of type $\lambda$ into $U_q(\mathfrak{g})$ and the set
\[
\{\, M(v) \subset U_q(\mathfrak{g}) \mid v \in F_S(\lambda) \,\}
\]
of $U_q(\mathfrak{l}_S)$-submodules of $U_q(\mathfrak{g})$ that are isomorphic to $V(\lambda)$.
\end{proposition}

\begin{proof}
Let $\mathcal{E}_\lambda$ denote the set of embeddings of type $\lambda$, and let
\[
\mathcal{M}_\lambda := \{\, M(v) \subset U_q(\mathfrak{g}) \mid v \in F_S(\lambda) \,\}.
\]

We define a map
\[
\Psi : \mathcal{E}_\lambda \longrightarrow \mathcal{M}_\lambda
\]
as follows: given an embedding $M \subset U_q(\mathfrak{g})$ with $M \simeq V(\lambda)$,
choose any nonzero element $v \in M$. Since $M$ is a $U_q(\mathfrak{l}_S)$-module,
the cyclic submodule generated by $v$ satisfies
\[
M(v) = \operatorname{ad}_L(U_q(\mathfrak{l}_S)) \cdot v \subseteq M.
\]
As $M$ is irreducible, we must have $M(v) = M$. Thus $\Psi(M) = M(v) = M$,
showing that $\Psi$ is well-defined.

Conversely, define
\[
\chi : \mathcal{M}_\lambda \longrightarrow \mathcal{E}_\lambda
\]
by sending a submodule $M(v)$ to itself, viewed as an embedding of type $\lambda$.

It is immediate that $\chi$ and $\Psi$ are inverse to each other:
for any $M \in \mathcal{E}_\lambda$, we have $\chi(\Psi(M)) = M$,
and for any $M(v) \in \mathcal{M}_\lambda$, we have $\Psi(\chi(M(v))) = M(v)$.

Therefore, $\chi$ and $\Psi$ define a bijection between $\mathcal{E}_\lambda$
and $\mathcal{M}_\lambda$, completing the proof.
\end{proof}

\begin{remark}
    The description in Proposition~\ref{prop:embedding bijection} admits an equivalent formulation in terms of generators. Define an equivalence relation on $F_S(\lambda)$ (see Lemma ~\ref{lem:equivalence relation}) by
\[
v \sim w \Longleftrightarrow M(v)=M(w).
\]
Then the assignment $v \mapsto M(v)$ induces a bijection
\[
F_S(\lambda)/\sim \, \, \longrightarrow \{ M(v) \subset U_q(\mathfrak{g}) \mid v \in F_S(\lambda) \}.
\]
In particular, embeddings of type $\lambda $ may be identified with equivalence classes of generators in $F_S(\lambda)$.
\end{remark}

\begin{remark}\label{lem:hwv}
Let $v \in F_S$. Then $M(v)$ contains a highest weight vector with respect to the adjoint action of $U_q(\mathfrak{l}_S)$.
\end{remark}

\begin{lemma}\label{lem:annhilation}
Let $\lambda \in \Lambda_S$. Then there exists $\widetilde{v} \in F_S(\lambda)$ such that
\[
\operatorname{ad}_L(E_i)(\widetilde{v})=0 \quad \forall i \in S.
\]
\end{lemma}

\begin{proof}
Let $v \in F_S(\lambda)$, so $M(v) \simeq V(\lambda)$. By the Remark~ \ref{lem:hwv}, $M(v)$ contains a highest weight vector $\widetilde{v}$. Then $M(\widetilde{v}) = M(v)$ by irreducibility. Thus $\widetilde{v} \in F_S(\lambda)$.
\end{proof}

\begin{proposition}
We have
\[
\Lambda_S
=
\left\{
\lambda \in P^+ \;\middle|\;
\exists\, v \in F_S \text{ such that }
\operatorname{ad}_L(E_i)(v)=0 \ \forall i \in S,
\text{ and } v \text{ has weight } \lambda
\right\}.
\]
\end{proposition}

\begin{proof}
We prove the two inclusions separately.

\medskip

\noindent

Let $\lambda \in \Lambda_S$. Then by Lemma~\ref{lem:annhilation}, there exists $v\in F_S(\lambda)$ such that  
\[
\operatorname{ad}_L(E_i)(v)=0 \quad \forall i \in S,
\]
Thus $\lambda$ belongs to the set on the right-hand side.

\medskip

\noindent

Conversely, suppose that there exists $v \in F_S$ such that
\[
\operatorname{ad}_L(E_i)(v)=0 \quad \forall i \in S,
\]
and $v$ has weight $\lambda$.

Then $M(v)$ is a finite-dimensional $U_q(\mathfrak{l}_S)$-module generated by
a highest weight vector of weight $\lambda$. It follows that $M(v)$ is a highest weight module of highest weight $\lambda$.

Since $M(v)$ is finite-dimensional, it is irreducible and hence isomorphic to
$V(\lambda)$. Therefore $\lambda \in \Lambda_S$.

\medskip

This proves the desired equality.
\end{proof}

\begin{definition}
Define the map
\[
\Phi : F_S \longrightarrow \{ \text{isomorphism classes of finite-dimensional } U_q(\mathfrak{l}_S)\text{-modules} \}
\]
by
\[
\Phi(v) = [M(v)],
\]
where $M(v) = \operatorname{ad}_L(U_q(\mathfrak{l}_S)) \cdot v$.
\end{definition}

\begin{definition}
For a finite-dimensional irreducible $U_q(\mathfrak{l}_S)$-module $V(\lambda)$, define the fiber
\[
\Phi^{-1}(V(\lambda)) := \{\, v \in F_S \mid M(v) \simeq V(\lambda) \,\}= F_S(\lambda).
\]
\end{definition}

\begin{proposition}\label{thm:fiber_collapse}
We assume that $S=\lbrace 1,2, \cdots, r\rbrace \setminus\lbrace x \rbrace$ is cominuscule.
Then for all $n \ge 0$, we have
\[
M(K_{-2n\varpi_x}) \simeq V(-\alpha_x).
\]
In particular,
\[
K_{-2n\varpi_x} \in \Phi^{-1}(V(-\alpha_x)) \quad \text{for all } n \ge 0.
\]
\end{proposition}

\begin{proof}
Fix $n \ge 0$ and consider $\mu=-2n\varpi_x$, and the element $K_\mu\in F_S$ due to (\cite{joseph2012quantum}, \S 7.1.3). Now by \cite[Proposition 1]{krahmer2004dirac}, the element
\[
x_1 := \operatorname{ad}_L(F_x)(K_\mu)
\]
is a highest weight vector in $F_S$ with highest weight $-\alpha_x$.

Since $x_1 \in M(K_\mu)$, we have
\[
M(x_1) \subseteq M(K_\mu).
\]
On the other hand, $x_1$ generates an irreducible highest weight module
of highest weight $-\alpha_x$, hence
\[
M(x_1) \simeq V(-\alpha_x).
\]

Since $M(K_\mu)$ is generated by $K_\mu$ and contains the highest weight vector $x_1$,
it follows that
\[
M(K_\mu) = M(x_1).
\]
Thus
\[
M(K_{-2n\varpi_x}) \simeq V(-\alpha_x),
\]
which proves the claim.
\end{proof}

\begin{corollary}\label{cor:infinite_fiber}
The fiber $\Phi^{-1}(V(-\alpha_x))$ contains
$\lbrace K_{-2n\varpi_x}\mid n\geq 0 \rbrace $.
As a consequence, $\lbrace K_{-2n\varpi_x}\mid n\geq 0 \rbrace \subseteq F_S(-\alpha_{x})$ for the cominuscule case.
\end{corollary}

\section{Structure of cyclic adjoint modules}

\subsection*{Partial order structure of cyclic adjoint modules}

We define a relation on $F_S$ by 
\[
w \le v \quad \Longleftrightarrow \quad M(w) \subseteq M(v) \quad \text{for} \,\,\, v,w\in F_S.
\]

Let
\[
\mathcal{P}_S := \{ M(v) \mid v \in F_S \}.
\]

The relation $\le$ defines a preorder on $F_S$, and induces a partial order on $\mathcal{P}_S$.

Since $M(v)$ is finite-dimensional, every strictly increasing chain of its submodules is finite. In particular, if $v \in F_S$ is such that $M(v)$ is irreducible, then $M(v)$ is minimal in $\mathcal{P}_S$.

\begin{remark}
    The partial order structure $(\mathcal{P}_S, \subseteq)$, in general, is not a lattice.
\end{remark}

\begin{proposition}\label{prop:hw_decomposition_refined}
Let $v \in F_S$. Then there exist highest weight vectors
$w_1, \dots, w_k \in M(v)$ such that each $M(w_i)$ is an irreducible
$U_q(\mathfrak{l}_S)$-submodule of $M(v)$ and
\[
M(v) = \sum_{i=1}^k M(w_i).
\]
\end{proposition}

\begin{proof}
We argue by induction on $\dim M(v)$.

Since $M(v)$ is finite-dimensional, Remark~\ref{lem:hwv}
 ensures the existence of a highest weight vector $w_1 \in M(v)$.
Then $M(w_1)$ is an irreducible submodule of $M(v)$.

If $M(w_1) = M(v)$, there is nothing to prove. Otherwise, consider the quotient
\[
\overline{M} := M(v)/M(w_1),
\]
which is finite-dimensional.

Applying Remark~\ref{lem:hwv} to $\overline{M}$, we obtain a highest weight
vector $\overline{w}_2 \in \overline{M}$. Let $w_2 \in M(v)$ be a preimage of
$\overline{w}_2$. Then $M(w_2)$ is an irreducible submodule of $M(v)$, and its image
in $\overline{M}$ coincides with the cyclic submodule generated by $\overline{w}_2$.

In particular,
\[
M(w_1) + M(w_2) \subseteq M(v).
\]

If $M(w_1) + M(w_2) = M(v)$, we are done. Otherwise, we repeat the argument
with the quotient by $M(w_1) + M(w_2)$.

Since the dimension strictly decreases at each step, the process terminates
after finitely many steps, yielding highest weight vectors
$w_1, \dots, w_k$ such that
\[
M(v) = \sum_{i=1}^k M(w_i).
\]
\end{proof}

\begin{proposition}\label{thm:poset_structure_correct}
Let $v \in F_S$, and consider the set
\[
[0, M(v)] := \{ M(w) \in \mathcal P_S \mid M(w) \subseteq M(v) \}.
\]
Then the minimal elements of $[0, M(v)]$ are precisely the irreducible
submodules of $M(v)$ of the form $M(w)$.
\end{proposition}

\begin{proof}
First, let $M(w) \subseteq M(v)$ be an irreducible submodule.
Then $M(w)$ has no proper nonzero submodules. Hence if
$M(u) \subseteq M(w)$, then either $M(u)=0$ or $M(u)=M(w)$.
Thus $M(w)$ is minimal in $[0, M(v)]$.

Conversely, let $M(w)$ be a minimal element in $[0, M(v)]$.
Suppose that $M(w)$ is not irreducible. Then it contains a proper
nonzero submodule $M(u) \subsetneq M(w)$.
Since $M(u) \subseteq M(v)$, we have $M(u) \in [0, M(v)]$,
which contradicts the minimality of $M(w)$.
Thus $M(w)$ must be irreducible.
\end{proof}

\begin{corollary}
The poset $[0, M(v)]$ encodes the inclusion relations among cyclic
submodules of $M(v)$. As a consequence of the Proposition ~\ref{thm:poset_structure_correct}, every cyclic adjoint module $M(v)\subset U_q(\mathfrak{g})$, there exists an embedding of type $\lambda$, namely $M(w)$ (for some $\lambda\in \Lambda_S$) such that $M(w)\subseteq M(v)$.
\end{corollary}


\bibliographystyle{plain}
\bibliography{name3}

\end{document}